\documentclass[11pt]{article}
\usepackage{fullpage}
\usepackage{fancyhdr}
\usepackage{amsmath,amssymb,amsthm}
\usepackage{amsmath,amscd}
\usepackage{mathtools}
\usepackage{mathrsfs}
\usepackage{tikz-cd}
\usepackage{authblk}
\usepackage{enumitem}

\usepackage[letterpaper, headsep=1.5cm, margin=1.2in]{geometry}
\usepackage{baskervald}

\newcommand{\Z}{\mathbb{Z}}

\newcommand{\Q}{\mathbb{Q}}

\newcommand{\Spec}{\text{Spec}}

\newcommand{\res}{\text{Res}}

\newcommand{\Fisoc}{\mathbf{F}- \mathbf{Isoc}}

\newcommand{\Mphinosigma}[1]{\textbf{M}\Phi^\nabla_{ #1}}
\newcommand{\Mphinosigmaf}[1]{\textbf{M}{\Phi^f}^\nabla_{ #1}}

\newcommand{\ang}[1]{\langle #1 \rangle}

\renewcommand {\bar}{\overline}

\newtheorem{theorem}{Theorem}[section]
\newtheorem{lemma}[theorem]{Lemma}

\newtheorem{definition}[theorem]{Definition}

\newtheorem{proposition}[theorem]{Proposition}

\def\frak{\relaxnext@\ifmmode\let\next\frak@\else
	\def\next{\Err@{Use \string\frak\space only in math mode}}\fi\next}
\def\goth{\relaxnext@\ifmmode\let\next\frak@\else
	\def\next{\Err@{Use \string\goth\space only in math mode}}\fi\next}
\def\frak@#1{{\frak@@{#1}}}
\def\frak@@#1{\noaccents@\fam\euffam#1}
\font\tengoth=eufm10
\newfam\gothfam \def\goth{\fam\gothfam\tengoth} \textfont\gothfam=\tengoth

\title{Log-decay $F$-isocrystals on higher dimensional varieties}
\author{Joe Kramer-Miller}

\date{}

\begin{document}

	\maketitle

	\newcommand{\Addresses}{{
			\bigskip
			\footnotesize
			
			Joe Kramer-Miller, \textsc{Department of Mathematics, University College London,
				Gower Street, London}\par\nopagebreak
			\textit{E-mail address}, Joe Kramer-Miller: \texttt{j.kramer-miller@ucl.ac.uk}

	}}
		\begin{abstract} 
			Let $k$ be a perfect field of positive characteristic
			and let $X$ be a smooth irreducible quasi-compact 
			scheme over $k$. The Drinfeld-Kedlaya theorem
			states that for an irreducible $F$-isocrystal
			on $X$, the gap between consecutive generic slopes
			is bounded by one. In this note we provide a new
			proof of this theorem. Our proof utilizes the
			theory of $F$-isocrystals with $r$-log decay. We
			first show that a rank one $F$-isocrystal with 
			$r$-log decay is overconvergent if $r<1$. Next, we
			establish a connection between slope gaps and
			the rate of log-decay of the slope filtration. 
			The Drinfeld-Kedlaya theorem then follows from
			a simple patching argument.

		\end{abstract}
	
	\section{Introduction}
	\subsection{Motivation}
		Let $k$ be a perfect field of positive characteristic
		and let $X$ be a smooth irreducible quasi-compact 
		scheme over $k$. When studying motives over $X$,
		one typically studies their $\ell$-adic realization
		for some $\ell\neq p$.
		These are lisse $\ell$-adic sheaves on $X$, which
		correspond to continuous $\ell$-adic representations
		of $\pi_1^{et}(X)$. While lisse $\ell$-adic sheaves
		are sufficient for studying the $\ell$-adic
		and archimedean properties of a motive, thus far they have
		been insufficient for studying $p$-adic questions. 
		For example, for a smooth proper fibration $f:Y\to X$, we know
		that the Frobenius eigenvalues of $R_i^{et}f_*(\Q_\ell)$
		at a closed point $x \in X$ has $\ell$-adic valuation zero, but
		there does not exist such a sweeping general statement
		about the $p$-adic valuations. In general, the $p$-adic
		valuations will change as $x$ varies.
		It is therefore natural to ask how the $p$-adic valuations
		 behave
		as one varies $x$ (i.e. how does the $p$-adic Newton polygon
		of the characteristic polynomial of the Frobenius vary).
		By considering the $p$-adic realization
		of a motive, which are $F$-isocrystals, there are several beautiful statements about the
		 variation
		of these Newton polygons. The first general result
		along these lines is due to Grothendeick, and says that
		the Newton polygon of a generic point lies below the
		Newton polygon of any specialization. 
		Another significant result
		is the de Jong-Oort purity theorem, which tells us
		these Newton polygons are constant on an open subscheme 
		and jump on a closed subscheme of codimension one. More
		recently, we have the Drinfeld-Kedlaya theorem (see \cite{Drinfeld-Kedlaya}). This
		theorem states that for an irreducible $F$-isocrystal,
		the gaps between slopes of the generic Newton polygon
		are bounded by one. The purpose of this
		article is to provide a new proof of this theorem using
		$F$-isocrystals with logarithmic decay. Along the way,
		we describe an interesting connection between the slope
		filtration and the log-decay condition.

		\subsection{Statement of the main result and proof outline}

		Let $M$ be either
		a convergent $F$-isocrystal or an overconvergent $F$-isocrystal 
		on $X$ whose rank is $n$ (see \S \ref{F-isocrystal review intro}). For any point $x\in X$,
		we may associate to $M|_x$ rational
		numbers $a_x^1(M),\dots, a_x^n(M) \in \Q$, with $a_x^1(M) \leq \dots \leq a_x^n(M)$, which we call
		the slopes of $M_x$ (see \cite{katz-slope_filtration}). Informally, one may think of the slopes
		as the $p$-adic valuations of the ``eigenvalues" of the Frobenius acting on $x^* M$. The Drinfeld-Kedlaya theorem is the following:
		
		\begin{theorem} \label{main result}
			Assume that $M$ is irreducible. Let $\eta \in X$ be 
			the generic point. Then 
			\begin{align*}
				|a_\eta^{i+1}(M) - a_\eta^i(M)|\leq 1,
			\end{align*}
			for each $i$.
		\end{theorem}

		In the first step of the proof of Theorem \ref{main result}, we study a rank one isocrystals $M$ on 
		$\mathbb{G}_{m,k}^n \times \mathbb{A}_k^m$ with $r$-log decay. In Theorem \ref{small log decay are almost regular}, we prove that
		if $r<1$ then a tensor power $M^{\otimes p^k}$ has a log-connection.
		When $M$ has a compatible Frobenius structure, we find that
		a higher tensor power extends to $\mathbb{A}_k^{n+m}$. This
		implies that the representation of
		$\pi_1^{et}(\mathbb{G}_{m,k}^n \times \mathbb{A}_k^m)$ corresponding
		to $M$ is potentially unramified at the coordinate planes, and thus $M$ is overconvergent
		by a result of Kedlaya (see \cite[Theorem 2.3.7]{kedlaya2011swan}).
		It would be interesting to know if all isocrystals on $\mathbb{G}_{m,k}^n \times \mathbb{A}_k^m$ with $r$-log-decay are overconvergent when $r<1$.

		For the second step of the proof, we consider an overconvergent
		$F$-isocrystal $M$ on $\mathbb{G}_{m,k}^n \times \mathbb{A}_k^m$ 
		whose Newton polygon remains constant for each point
		$x \in \mathbb{G}_{m,k}^n \times \mathbb{A}_k^m$. Thus $M$
		obtains a slope filtration in the category of convergent
		$F$-isocrystals. That is, if
		$a_x^i(M)<a_x^{i+1}(M)$, there exists a convergent
		sub-$F$-isocrystal $M_i$ of $M$, such that $a_x^j(M_i)=a_x^j(M)$ for
		all $j\leq i$ and all $x \in \mathbb{G}_{m,k}^n \times \mathbb{A}_k^m$.
		In Proposition \ref{Frobenius structure of u-r has $r$-log-decay}
		we prove that $M_i$ has $r_i$-log-decay, where $r_i=\frac{1}{a_x^{i+1}(M)-a_x^i(M)}$. Combining
		this with the results of the previous paragraph, we
		see that $M_i$ is overconvergent if $a_x^{i+1}-a_x^i(M)>1$.
		
		The final step involves a geometric patching argument.
		We first consider a generically \'etale alteration
		$Y \to X$ with compactification $\bar{Y}$ such that
		$\bar{Y}-Y$ is a normal crossing divisor.
		Using the results of \cite{MR2092132} we find finite \'etale maps
		locally on $Y$ onto $\mathbb{G}_{m,k}^n \times \mathbb{A}_k^m$ 
		this allows us to use the ideas of the previous paragraphs.

		\subsection{Relationship with previous approaches}
		The proof of Drinfeld and Kedlaya
		in \cite{Drinfeld-Kedlaya} can be summarized as follows:
		first they prove that if $U\subset X$ is a dense open subscheme,
		the restriction functor from convergent $F$-isocrystals on $X$ to $F$-isocrystals
		on $U$ is fully faithful. This builds upon several other 
		difficult fully faithful results due to Kedlaya and Shiho (see  \cite{Kedlaya-fully_faithful}, \cite[Theorem 5.1]{kedlaya-semistable1}, and \cite{Shiho-purity_for_overconvergence}). Let $M$ be an $F$-isocrystal on $X$. 
		Let $U \subset X$
		be the locus of points where $a_\eta^i(M)=a_x^i(M)$. 
		When we restrict $M$ to $U$, we obtain a slope filtration.
		In particular, $M|_U$ corresponds to an element of $\text{Ext}^1(M_1,M_2)$,
		where the slopes of $M_1$ are less than those of $M_2$. 
		In \cite{Drinfeld-Kedlaya}, they show that $\text{Ext}^1(M_1,M_2)$ is trivial when the
		smallest slope of $M_2$ is greater than one plus the largest
		slope of $M_1$. Therefore a gap larger than one in the slopes means $M|_U=M_1\oplus M_2$. This decomposition
		provides idempotent morphisms from $M|_U\to M|_U$, which extend to $M$
		by the aforementioned fully faithfulness result.
		
		Our proof can be viewed as orthogonal to the Drinfeld-Kedlaya
		approach in two facets. First, instead of restricting to
		the constant locus of the Newton polygon, we prove that
		the Newton polygon remains constant along all of $X$. 
		This allows us to completely bypass
		any sort of fully faithfulness result when $M$ is convergent and 
		only use the fully faithfulness of
		the overconvergent $F$-isocrystals to
		convergent $F$-isocrystals functor when $M$ is overconvergent (see \cite{Kedlaya-fully_faithful}). Second is the proof
		that $\text{Ext}^1(M_1,M_2)$ is trivial, which can be traced
		back to Kedlaya's thesis (see \cite[Theorem 5.2.1]{Kedlaya-thesis}).
		The underlying idea is that when the gaps between the slopes
		are larger than one, the connection is preserved at the corresponding
		step of the descending slope filtration due to de Jong (see \cite{deJong-homomorphisms_of_BT}). The descending slope 
		filtration only exists over some purely inseperable pro-cover, but
		using the connection it is possible to descend part of the
		filtration to the original base. This is
		in contrast to our proof, where we show that
		the pertinent step of the ascending slope filtration descends
		using the notion of $r$-logarithmic decay. 
	
		It is also worth mentioning that Drinfeld and Kedlaya
		assume that $M$ is indecomposible. This is decidedly stronger
		than Theorem \ref{main result}, where we assume $M$ is irreducible. However, for the applications
		in \cite{Drinfeld-Kedlaya} and the other applications we are
		aware of (e.g. \cite{Krishnamoorthy-Pal}), Theorem \ref{main result} is sufficient. Of course, one could apply the $\text{Ext}^1$ result
		in \cite{Drinfeld-Kedlaya} together with Theorem \ref{main result}
		to obtain this more general result.
		
		Finally, let us mention previous work of the author, where
		we proved Theorem \ref{main result} for curves over a finite field
		(see \cite[Corollary 7.4]{JKM-F-isoc_geometric_origin}). 
		In this work, Theorem \ref{main result} was a corollary 
		of a difficult monodromy theorem for rank one convergent $F$-isocrystals
		and an analysis of the slope filtration. In particular,
		using a monodromy theorem (\cite[Corollary 4.16]{JKM-F-isoc_geometric_origin}) 
		and class field theory, we
		showed that for $r<1$, a rank one $F$-isocrystal with
		$r$-log-decay is overconvergent. This is the same
		as Theorem \ref{small log decay $F$-isocrystals are overconvergent}.
		However, we maintain that the present approach is preferable and necessary. First, the proof of Theorem \ref{small log decay $F$-isocrystals are overconvergent} in this article relies on a study of the
		underlying differential equation. This elementary approach
		completely bypasses the technical monodromy theorem used in
		\cite{JKM-F-isoc_geometric_origin}. It is also amenable to
		the higher dimensional situation, where ramification theory
		is much more technical. Second, in this article we may
		take our ground field to be any perfect field $k$. Lastly,
		in this paper we deal with varieties of arbitrary dimension.
		Although one could use the Lefschetz theorem for $F$-isocrystals,
		due to Abe and Esnault (see \cite{Abe-Esnault}), to obtain
		results on higher dimensional varieties, this only works
		for finite fields. It also has the downside of being less
		direct of than the proof presented here.
		
		\subsection{A remark on logarithmic decay $F$-isocrystals}
		The notion of $F$-isocrystals with a log-decay Frobenius structure
		was introduced by Dwork-Sperber and plays a prominent
		role in Wan's work on unit-root $L$-functions (see \cite{Dwork-Sperber}, \cite{Wan-meromorphic}, and \cite{WanDworksconjecture}).
		The log-decay condition for Frobenius structures arise naturally in the study of unit-root $F$-isocrystals.
		However, they only study $F$-isocrystals over $\mathbb{G}_{m,k}^n$.
		In \cite{JKM-F-isoc_geometric_origin}, the author studied
		$F$-isocrystals with log-decay in both
		the Frobenius structure and the differential equation over curves. We
		studied the rate of log-decay of the slope filtration and
		the monodromy properties of $F$-isocrystals with log-decay. 
		In the present article, we utilize the log-decay notion
		for a general variety $X$ in a somewhat ad-hoc manner.
		We find an alteration $Y\to X$, whose compactification $\bar{Y}$
		is smooth and $\bar{Y}-Y$
		is a normal crossing divisor $D$. We then we cover $\bar{Y}$
		with open subschemes that admit an etale map onto $\mathbb{A}^n$,
		that take $D$ to coordinate planes. This lets us use an explicit
		definition of $r$-log-decay for $F$-isocrystals
		on $\mathbb{G}_{m,k}^n$. Although sufficient for our applications,
		it is not clear if the property of having $r$-log-decay
		is intrinsic to an $F$-isocrystal on $X$. What if
		we choose a different alteration $Y$ or find different
		etale maps onto $\mathbb{A}^n$? It would be interesting
		to either find an intrinsic definition
		of logarithmic decay or to prove that the ad-hoc notion
		used in this article is intrinsic.

	\section{Rings of functions on polyannuli}
	\label{polyannuli rings}
		Let $K$ be $W(k)[\frac{1}{p}]$ and let $\sigma$ the lift
		of the Frobenius morphism on $k$.
		Let $n>0$ and let $m\geq 0$. Consider indeterminates $T_1,\dots,T_{n+m}$. We then define
		\begin{align*}
		\mathcal{A}&=K\ang{T_1^\pm, \dots T_n^\pm, T_{n+1},\dots, T_{n+m}} \\
		\mathcal{A}_{(i)} &=K\ang{T_1^\pm, \dots T_{i-1}^\pm, T_{i+1}^\pm, \dots T_n^\pm, T_{n+1},\dots, T_{n+m}} \\
		\end{align*}
		We may extend $\sigma$ to $\mathcal{A}$ by having $\sigma$
		send $T_i$ to $T_i^p$. We let $|. |_1$ be the Gauss norm on $\mathcal{A}$. Let $x(T) \in \mathcal{A}$. For each $i=1,\dots, n$
		we may write 
		\begin{align*}
		x(T) = \sum_{d\in\Z} a_dT_i^d,
		\end{align*}
		with $a_d \in \mathcal{A}_{(i)}$ and $|a_d|_1 \to 0$ as
		$|d| \to \infty$. We refer to this as the $T_i$-adic
		expansion of $x(T)$. Using this expansion, we define
		some truncations of $x(T)$:
		\[ W^{(i)}_{<}(x(T)) = \sum_{d<-1} a_dT_i^d ~~~~ W^{(i)}_{\geq}(x(T)) = \sum_{d\geq 0} a_d T_i^d. \]
		Note that $W^{(i)}_<(W^{(j)}_\geq(x(T))) = W^{(j)}_\geq (W^{(i)}_<(x(T)))$.
		Next, we define the $T_i$-adic $j$-th partial
		valuation as
		\begin{align*}
		w^{(i)}_j(x(T)) &= \min_{ v(a_d)\leq j} \{d\}.
		\end{align*}
		Using these partial valuations, we define the 
		ring of overconvergent functions and the ring of $r$-log-decay functions:
		\begin{align*}
		\mathcal{A}^\dagger &= \Bigg \{ x(T) \in \mathcal{A} ~\Bigg | \begin{array} {l} \text{ there exists $c>0$ such that for }i=1,\dots, n \\ \text{ and $j \gg 0$, we have } w^{(i)}_j(x(T)) \geq -cj \end{array} \Bigg \} \\
		\mathcal{A}^r &= \Bigg \{ x(T) \in \mathcal{A} ~\Bigg | \begin{array} {l} \text{ there exists $c>0$ such that for }i=1,\dots, n \\ \text{ and $j \gg 0$, we have } w^{(i)}_j(x(T)) \geq -cp^{rj} \end{array} \Bigg \}.
		\end{align*}
		Note that $\sigma$ restricts to endomorphisms of $\mathcal{A}^\dagger$
		and $\mathcal{A}^r$. For $r<1$ and $x \in \mathcal{A}^r$, the 
		$T_i$-adic primitive of $W^{(i)}_< (x)$ converges to an element
		of $\mathcal{A}^{r+1}$:
		\begin{align*}
		\int W^{(i)}_< (x(T)) dT_i &= \sum_{d<-1} a_d\frac{T_i^{d+1}}{d}.
		\end{align*}
		We define the $T_i$-adic residue to be
		\begin{align*}
		\res_i(x(T)) &= a_{-1}. 
		\end{align*}
	\section{$F$-isocrystals}
	\label{F-isocrystal review intro}
	Let $X$ be a smooth irreducible quasi-compact scheme over $k$.
	We will freely use the notion of convergent and overconvergent $F$-isocrystals. For a high lever overview, we recommend 
	\cite{Kedlaya-notes_on_isocrystals}.
	We let $\Fisoc^\dagger(X)$ denote the category of
	overconvergent $F$-isocrystals on $X$ (see \cite{Kedlaya-notes_on_isocrystals} for precise definitions)
	and we $\Fisoc(X)$ denote the category of convergent $F$-isocrystals
	on $X$. For a dense open immersion $U\subset X$ we let
	$\Fisoc(U,X)$ denote the category of $F$-isocrystals
	on $U$ overconvergent along $X-U$. For any finite extension $E$
	of $\Q_p$ we let $\Fisoc(X) \otimes E$ (resp $\Fisoc^\dagger(X) \otimes E$, $\Fisoc(U,X)\otimes E$) denote the category whose objects are
	objects in $\Fisoc(X)$ (resp $\Fisoc^\dagger(X)$ , $\Fisoc(U,X)$) with a $\Q_p$-linear action of $E$. 
	Given an open subscheme $V \subset U$
	and $W \subset X$, such that $V \subset W$ is an open immersion, there is a natural restriction functor
	$\Fisoc(U,X) \to \Fisoc(V,W)$. If $M$ is an object of $\Fisoc(U,X)$,
	we refer to the image of $M$ in $\Fisoc(V,W)$ as the restriction
	of $M$ to the pair $(V,W)$. 
	
	Now let $M$ be an object of $\Fisoc(X)$. For any $x \in X$
	we define $b_x^i(M)=a_x^1(M)+\dots+a_x^i(M)$ and let $NP_x(M)$
	be the lower convex hull of the vertices $(i,b_x^i(M))$. 
	If $(i,y)$ is a vertex of $NP_x(M)$ for all $x\in X$, then
	by a theorem of Katz there exists a rank $i$ subobject $M_i$ of $M$ in $\Fisoc(X)$
	such that $a_x^j(M_i)=a_x^j(M)$ for all $j\leq i$.

		\subsection {$F$-isocrystals on $\mathbb{G}_{m,k}$ as $(\sigma^f,\nabla)$-modules over $\mathcal{A},\mathcal{A}^r,$ and $\mathcal{A}^\dagger$}
		When $U=\mathbb{G}_{m,k}^n\times \mathbb{A}_k^m$ and $X=\mathbb{A}_k^{n+m}$,
		we may view objects of $\Fisoc(U)$ and $\Fisoc(U,X)$
		as differential equations over the rings introduced in 
		\S \ref{polyannuli rings} with a compatible Frobenius structure.
		Let $R$ be either $\mathcal{A}^\dagger$, $\mathcal{A}^r$, or
		$\mathcal{A}$. 
		\begin{definition}
		
		A $\sigma^f$-module is a locally free $R$-module $M$
		equipped with a $\sigma^f$-semilinear endomorphism $\varphi: M \to M$
		whose linearization is an isomorphism.  More precisely,
		we have $\varphi(am)=\sigma^f(a)\varphi(m)$ for $a\in R$
		and $\varphi: R \otimes_{\sigma^f} M \to M$
		is an isomorphism.  
	\end{definition}
	\begin{definition}
		Let $\Omega_R$ be the module of differentials of $R$ over $K$. Let $\delta_T: R \to \Omega_R$ to be
		the exterior derivative.  A $\nabla$-module over $R$
		is a locally free $R$-module $M$ equipped with a connection.  That is,
		$M$ comes with a $K$-linear map $\nabla: M \to \Omega_R$
		satisfying the Liebnitz rule: $\nabla(am) = \delta_T(a)m + a\nabla(m)$.
	\end{definition}

	\begin{definition} By abuse of notation, define
		$\sigma^f: \Omega_R \to \Omega_R$ be the map induced 
		by pulling back the differential along $\sigma^f$.  
		A $(\sigma^f,\nabla)$-module is an
		$R$-module $M$ that is both a $\sigma^f$-module and a $\nabla$-module
		with the following compatibility condition:
		\begin{center}
			\begin{tikzcd}
				
				M \arrow{d}{\sigma^f} \arrow{r}{\nabla} &
				M \otimes \Omega_R \arrow{d}{\sigma^f \otimes \sigma^f} \\
				
				M \arrow{r}{\nabla} & M \otimes \Omega_R .
			\end{tikzcd}
		\end{center}
	\end{definition}

	\noindent We denote the category of $(\sigma^f,\nabla)$-modules over
	$R$ by $\Mphinosigmaf{R}$. We obtain functors $\Mphinosigmaf{\mathcal{A}^\dagger} \to \Mphinosigmaf{\mathcal{A}}$
	and $\Mphinosigmaf{\mathcal{A}^r} \to \Mphinosigmaf{\mathcal{A}}$
	by base changing to $\mathcal{A}$. We say that an object
	$M$ of $\Mphinosigmaf{\mathcal{A}}$ is overconvergent
	(resp. has $r$-log-decay) if it lies in the essential image
	of $\Mphinosigmaf{\mathcal{A}^\dagger} \to \Mphinosigmaf{\mathcal{A}}$
	(resp. $\Mphinosigmaf{\mathcal{A}^r} \to \Mphinosigmaf{\mathcal{A}}$).
	There are equivalences of categories
		\begin{align*}
	\Mphinosigmaf{\mathcal{A}} &\longleftrightarrow \Fisoc(\mathbb{G}_{m,k}^n\times \mathbb{A}_k^m) \otimes \Q_{p^f} \\
	\Mphinosigmaf{\mathcal{A}^\dagger} &\longleftrightarrow \Fisoc^\dagger(\mathbb{G}_{m,k}^n\times \mathbb{A}_k^m, \mathbb{A}_k^{n+m})\otimes \Q_{p^f} .
	\end{align*}
	There are natural functors $\epsilon_f:\Mphinosigma{\mathcal{A}} \to \Mphinosigmaf{\mathcal{A}}$
	and $\epsilon_f^\dagger:\Mphinosigma{\mathcal{A}^\dagger} \to \Mphinosigmaf{\mathcal{A}^\dagger}$,
	which are obtained by iterating the Frobenius map $f$ times.
	
	\section{Connections on polyannuli with $r$-log decay for $r<1$}
		In this section we study rank one $\nabla$-modules
		with small rates of logarithmic decay.
		\begin{lemma} \label{lemma: integrable connection forces poles}
			Let $M$ be an integrable $\nabla$-module over $\mathcal{A}$ or $\mathcal{A}^r$ with rank one. Let $e$ be a basis of $M$ and write
			\begin{align*}
			\nabla(e) &= f_1(T)dT_1 + \dots  + f_{n+m}(T)dT_{n+m}.
			\end{align*}
			We have $\res_i(f_i) \in K$ for $i=1,\dots,n$. Furthermore, if $W^{(i)}_{<}(f_i)=0$ then
			for each $j$ we have $W^{(i)}_{<}(f_j)=0$.
		\end{lemma}
		\begin{proof}
			Since $\nabla$ is integrable we know $\partial_i f_j= \partial_j f_i$. The lemma follows immediately.
		\end{proof}

		\begin{proposition} \label{small log decay are almost regular}
			Let $M$ be a rank one integrable $\nabla$-module over $\mathcal{A}^r$ for some $r<1$. Then there exists $m$ such that 
			$M^{\otimes p^m}$ has regular singularities.
		\end{proposition}

		\begin{proof}
			Let $e$ be a basis of $M$ and let $c_{1,1},\dots,c_{1,{n+m}} \in \mathcal{A}^r$
			with $\nabla(e)=\sum c_{1,i} dT_i \otimes e$. Since $r<1$ we know that 
			\begin{align*}
				h_1 =     \int W^{(1)}_{<}(c_{1,1} )dT_1 
			\end{align*}
			converges in $\mathcal{A}$. Thus for $\tau$ sufficiently large we may consider
			the basis
			$e_1=\exp(p^\tau h_1)e^{\otimes p}$ of $M^{\otimes p^\tau} \otimes \mathcal{A}$. We have 
			\begin{align*}
				\nabla^{\otimes p^\tau} (e_1) &= \sum_{i=1}^d c_{2,i} dT_i \otimes e_1 \\
				c_{2,i} &= \partial_i h_1 + p^\tau c_{1,i}.
			\end{align*}
			By our definition of $h_1$ we know that $W^{(1)}_<(c_{2,1})=0$,
			so by Lemma \ref{lemma: integrable connection forces poles}
			we have $W^{(1)}_<(c_{2,i})$ is zero for each $i$. From Lemma \ref{lemma: integrable connection forces poles} we know that $c_{2,2} = W^{(1)}_\geq (c_{2,2})$. 
			As $W^{(1)}_\geq (\partial_i h_1)=0$ this gives
			\begin{align*}
				W^{(2)}_<(c_{2,2}) &= W^{(2)}_<(W^{(1)}_\geq (c_{2,2})) \\
				&= p^\tau W^{(2)}_<(W^{(1)}_\geq (c_{1,2})),
			\end{align*}
			which is contained in $\mathcal{A}^r$. In particular
				\begin{align*}
			h_2 =  \int W^{(2)}_{<}(c_{2,2} )dT_2
			\end{align*}
			converges in $\mathcal{A}$. After increasing $\tau$, we may consider the basis $e_2=\exp(h_2) e_1$ of $M^{\otimes p^\tau } \otimes \mathcal{A}$. We have
			\begin{align*}
			\nabla^{\otimes p^\tau } (e_2) &= \sum_{i=1}^d c_{3,i} dT_i \otimes e_2 \\
			c_{3,i} &= \partial_i (h_1+h_2) + p^\tau c_{1,i}.
			\end{align*}
			Note that $W^{(j)}_<(c_{3,i})=0$ for $j<3$ and all $i$
			from Lemma \ref{lemma: integrable connection forces poles}.
			In particular $W^{(1)}_\geq(W^{(2)}_\geq(c_{3,3}))=c_{3,3}$.
			Since the truncation operators commute with each other
			and $W^{(j)}_\geq (\partial_i h_j)=0$ for $j=1,2$, we find as above
			that  
			\begin{align*}
				W^{(3)}_< (c_{3,3}) &= p^\tau  W^{(3)}_<(W^{(1)}_\geq(W^{(2)}_\geq(c_{1,3}))),
			\end{align*}
			which is contained in $\mathcal{A}^r$. This allows us
			to define $h_3$. The proposition follows from repeating this process.
		\end{proof}
	
		\begin{proposition} \label{ small log decay F-isocrystal extends after tensor power}
			Let $M$ be a $(\sigma^f,\nabla)$-module over $\mathcal{A}^r$
			for some $r<1$. 
			Then a tensor power of $M$ extends to $K\ang{T_1,\dots, T_{n+m}}$.
		\end{proposition}
		\begin{proof}
			By Proposition \ref{small log decay are almost regular}
			we may assume that $M$ has regular singularities.
			This means $M=e\mathcal{A}^r$ and
			\begin{align*}
				\nabla(e) &= f_1dT_1 +\dots f_{n+m} dT_{n+m},
			\end{align*}
			where $f_i =\frac{c_i}{T_i} + g_i$ with $c_i \in K$ and 
			$g_i \in T_i K\ang{T_1,\dots, T_{n+m}}$. Let $a \in \mathcal{A}^r$
			satisfy
			$\sigma^f(e)=ae$. The compatibility between $\sigma^f$ and $\nabla$
			imply
			\begin{align} \label{connection frob compatibility}
				T_i\frac{\partial_i a}{a} &= qf_i^{\sigma^f} - f_i.
			\end{align}
			This gives $\res(\frac{\partial_i a}{a})=qc_i^{\sigma^f} - c_i$.
			This residue is an integer $n_i$, so we have $c_i = \frac{n_i}{q-1}$.
			It follows that $M^{\otimes(q-1)}$ has a solution,
			and thus $\nabla$ extends to $K\ang{T_1,\dots,T_{n+m}}$.
			The compatibility between $\sigma^f$ and $\nabla$
			implies that the Frobenius also extends
			to $K\ang{T_1,\dots, T_{n+m}}$.
		\end{proof}

		\begin{theorem} \label{small log decay $F$-isocrystals are overconvergent}
			Let $M$ be a rank one object of $\Fisoc(\mathbb{G}_m^{n}\times \mathbb{A}_k^m,\mathbb{A}_k^{n+m})$, so that
			we may regard $M$ as $(\sigma,\nabla)$-module over $\mathcal{A}$. 
			If the connection descends to $\mathcal{A}^r$ for some $r<1$ 
			then
			$M$ is overconvergent.
		\end{theorem}
	
		\begin{proof}
			We may assume that $M$ is unit-root, and therefore corresponds
			to a $p$-adic character $\rho: \pi_1(\mathbb{G}_m^{n}\times \mathbb{A}_k^m) \to E^\times$,
			where $E$ is a finite extension of $\Q_p$. By 
			Proposition \ref{ small log decay F-isocrystal extends after tensor power} we $M^{\otimes \tau}$ extends to an $F$-isocrystal on
			$\mathbb{A}_k^{n+m}$, meaning that $\rho^{\otimes \tau}$ extends to a representation of $\pi_1(\mathbb{A}_k^{n+m})$. This implies
			$\rho$ is potentially unramified as in \cite[Definition 2.3.6]{kedlaya2011swan}. By \cite[Theorem 2.3.7]{kedlaya2011swan} we know that $M$ is overconvergent
			along the divisor $T_1\dots T_n=0$
			in $\mathbb{A}_k^{n+m}$.
		\end{proof}
	
		\section{Slope filtrations and log-decay}
		Let $M$ be a free $(\sigma,\nabla)$-module over $\mathcal{A}^\dagger$
		of rank $d$.
		We will assume that the Newton polygon of $x^* M$ remains
		constant as we vary over all points $x:\Spec(k_0)\to \mathbb{G}_{m,k}^n\times \mathbb{A}^{n}$,
		that the slopes are non-negative, and that the slope zero
		occurs exactly once. In particular, when we look at the image
		of $M$ has a rank one subobject $M^{u-r}$ in the category $\Mphinosigma{\mathcal{A}}$. 
		The main result of this section is:
		\begin{theorem} \label{theorem on log-decay of subcrystal}
			Let $s$ be the smallest nonzero slope of $M$
			and let $r=\frac{1}{s}$. There exists $f$ such that
			$\epsilon_f(M^{u-r})$ has $r$-log-decay.
		\end{theorem}

		\noindent We first introduce some auxiliary subrings of $\mathcal{A}^r$
		and $\mathcal{A}^\dagger$. We define
		\begin{align*}
		\mathcal{A}^{r,c} &= \Bigg \{ x(T) \in \mathcal{A} ~\Bigg | \begin{array} {l} w^{(i)}_k(x(T)) \geq -cp^{rk} \text{ for }k>0
		 \\   \text{ and }w^{(i)}_0(x(T))\geq 0 \text{ for }i=1,\dots, d \end{array} \Bigg \},\\
		 \mathcal{A}^{\dagger,r,c} &= \mathcal{A}^{r,c} \cap \mathcal{A}^\dagger.
		\end{align*}
		Note that $\mathcal{A}^{r,c}$ is $p$-adically complete, unlike
		$\mathcal{A}^r$. When $R$ is any of these rings, we let
		$\mathcal{O}_R$ denote the ring of elements in $R$ whose
		Gauss norm is less than or equal to one. 
		Now for $f$ large enoguh, there exists $\omega \in K$ with $v_{p^f}(\omega) = s$. The following
		Lemma follows from the definition of $\mathcal{A}^{r,c}$.
		\begin{lemma}
			We have the following:
			\begin{enumerate}
				\item For $x\in \mathcal{A}^{r,c}$ we have $x^{\sigma^f} \in \mathcal{A}^{r,p^f c}$ and $\omega x \in \mathcal{A}^{r,p^{-f}c}$.
				\item Let $x \in \mathcal{A}$ with $w_0(x)\geq 0$. If $\omega x \in \mathcal{A}^{r,c}$ then $x \in \mathcal{A}^{r,p^{-f}c}$.
			\end{enumerate}
		\end{lemma}

		By our
		assumptions on the slope of $M$, there exists a basis of
		$M$ whose Frobenius is given by a matrix
		\begin{align*}
			A &=\begin{pmatrix} A_{1,1} & \omega A_{1,2} \\ \omega A_{2,1} & \omega A_{2,2} \end{pmatrix},
		\end{align*}
		where $A_{i,j}$ are matrices with entries in $\mathcal{O}_{\mathcal{A}^\dagger}$ and  $A_{1,1}\in \mathcal{O}_{\mathcal{A}^\dagger}^\times$. The connection is
		given by the differential matrices
		\begin{align*}
			C_1 dT_1 + \dots + C_{n+m} dT_{n+m},
		\end{align*}
		where the $C_i$ are $d\times d$ matrices with entries 
		in $\mathcal{O}_{\mathcal{A}^\dagger}$. For $a \in \mathcal{A}$ we let $K_{u,v}(a)$ be the matrix with
		$a$ in the $(u,v)$-entry and zero elsewhere. We let
		$L_{u,v}(a)=1_d + K_{u,v}(a)$, where $1_d$ is the $d\times d$
		identity matrix.
		
		\begin{lemma} \label{entry growth lemma}
			After a change of variables, we may assume the following
			holds for all $i$:
			\begin{enumerate}
				\item $w_0^{(i)}(A_{1,1}^{-1}) \geq 0$ \label{fixing our basis 1}
				\item $w_0^{(i)}(A_{1,2})\geq 0$ and $w_0^{(i)}(A_{2,2}) \geq 0$.\label{fixing our basis 2}
			\end{enumerate}
		\end{lemma}
		\begin{proof}
			To prove the first claim, it is enough to 
			prove $w_0^{(i)}(A_{1,1})\leq 0$. Let $C_i=\begin{pmatrix} T_i^{k} & 0 \\ 0 & 1_{d-1} \end{pmatrix}$. For $k$ large enough,
			we see that $C_i A C_i^{-\sigma^f}$
			satisfies the desired property. For the second claim
			we set $D_i = \begin{pmatrix} 1 & 0 \\ 0 & T_i^{-k} 1_{d-1} \end{pmatrix}.$ When $k$ is large enough,
			we find that $D_iA D_c^{-\sigma^f}$ satisfies the second property without changing
			the $(1,1)$-entry. 
		\end{proof}
		\noindent By Lemma \ref{entry growth lemma} we know that for $c$ sufficiently
		large the entries of $A_{1,1}^{-1},\omega A_{2,1}, \omega A_{1,2}$ and $\omega A_{2,2}$ are contained
		in $\mathcal{O}_{\mathcal{A}^{\dagger,r,c}}$. 
		\begin{lemma}\label{log decay properties}
			Let $x,y \in \mathcal{O}_{\mathcal{A}^\dagger}$. Assume that $x^{-1},\omega y \in \mathcal{O}_{\mathcal{A}^{\dagger,r,c}}$
			and $w_0(y)\geq 0$. Then 
			\begin{enumerate}
				\item We have $x^{-\sigma^f}\omega y \in \mathcal{O}_{\mathcal{A}^{\dagger,r,c}}$. \label{precise-log-growth basic lemma1}
				\item If $\omega|y$ then $(x + y)^{-1} \in \mathcal{O}_{\mathcal{A}^{\dagger,r,c}}$.
				\label{precise-log-growth basic lemma2}
			\end{enumerate}
		\end{lemma}
	\begin{proof}
		We know that $y \in \mathcal{O}_{\mathcal{A}^{\dagger,r,p^f c}}$,
		which means $x^{-\sigma^f}y\in \mathcal{O}_{\mathcal{A}^{\dagger,r,p^f c}}$. This implies $x^{-\sigma^f} \omega y \in \mathcal{O}_{\mathcal{A}^{\dagger,r,c}}$.
		Since $\mathcal{O}_{\mathcal{A}^{r,c}}$
		we know the geometric series $(1 + x^{-1}y)^{-1}$ is contained
		in $\mathcal{O}_{\mathcal{A}^{r,c}}$. As $1+x^{-1}y \in \mathcal{O}_{\mathcal{A}^\dagger}^\times$ we know that 
		$(1 + x^{-1}y)^{-1} \in \mathcal{O}_{\mathcal{A}^{\dagger,r,c}}$.
		Thus $(x+y)^{-1}=x^{-1}(1+x^{-1}y)^{-1}$ is contained in
		$\mathcal{O}_{\mathcal{A}^{\dagger,r,c}}$.
		
	\end{proof}

	\begin{proposition} \label{Frobenius structure of u-r has $r$-log-decay}
		There exists $N =\begin{pmatrix} 1 & 0 \\ N_{2,1} & 1_{d-1} \end{pmatrix}$
		where $N_{2,1}$ has entries in $\mathcal{O}_{\mathcal{A}^{\dagger,r,c}}$
		and $N A N^{-\sigma^f}$ is of the form $A' =\begin{pmatrix} A_{1,1}' & \omega A_{1,2} \\ 0 & A_{2,2}' \end{pmatrix}$.
	\end{proposition}
	\begin{proof}
		We will
		show inductively that there exists $N_k=\begin{pmatrix} 1 & 0 \\ N_{2,1,k} & 1_{d-1} \end{pmatrix}$ such that :
		\begin{enumerate}
			\item $A_k=N_k A N_k^{-\sigma^f}$ is of the form
			$\begin{pmatrix} A_{1,1,k} & \omega A_{1,2} \\ \omega^k A_{2,1,k} & \omega A_{2,2,k} \end{pmatrix}$ \label{matrix condition 1}
			\item  The entries of $A_{1,1,k}^{-1}, \omega A_{1,2,k},\omega A_{2,2,k},$ and $\omega^k A_{2,1,k}$ \label{matrix condition 2}
			are contained in $\mathcal{O}_{\mathcal{A}^{\dagger,r,c}}$
			\item We have $w_0^{(i)}(A_{1,1,k}^{-1}) \geq 0$, $w_0^{(i)}(A_{1,2,k})\geq 0$ and $w_0^{(i)}(A_{2,2,k}) \geq 0$.\label{matrix condition 3}
			\item For all $k$ we have $N_k \equiv N_{k-1} \mod \omega^k$. \label{matrix condition 4}
		\end{enumerate}
	The result will follow by taking the limit of the $N_k$ as $k\to\infty$.
	When $k=1$ this follows from Lemma \ref{entry growth lemma}.
	Now let $k>1$ and assume $N_k$ exists. We define
	\begin{align*}
		N_k &= \begin{pmatrix} 1 & 0 \\ -A_{1,1,k}^{-1}\omega^k A_{2,1,k} & 0 \end{pmatrix},
	\end{align*}
	and set $A_{k+1}=N_k A_{k} N_k^{-\sigma^f}$. It is immediate that
	\ref{matrix condition 1}, \ref{matrix condition 3}, and \ref{matrix condition 4} are satisfied. We can verify \ref{matrix condition 2}
	using Lemma \ref{log decay properties}.
	
	\end{proof}
		
	\begin{proof}[Proof of Theorem \ref{theorem on log-decay of subcrystal}]
		Let $N$ and $A'$ be as in Proposition \ref{Frobenius structure of u-r has $r$-log-decay}. After changing basis by $N$, the connection
		is given by the matrix of $1$-forms:
		\begin{align*}
			D_1dT_1 + \dots +
			D_{n+m}dT_{n+m},
		\end{align*}
		where $D_i=\partial_i M + C_iM$. In particular, the $D_i$ has entries
		in $\mathcal{O}_{\mathcal{A}^r}$. Compatibility between
		the connection and Frobenius give the relation:
		\begin{align*}
			\partial_i A' + D_iA' &= qA'D_i^{\sigma^f}. 
		\end{align*}
		Write $D_i=\begin{pmatrix} R_i & S_i \\ U_i & V_i \end{pmatrix}$,
		where $V$ is a $(d-1)\times(d-1)$-matrix.
		Continuing with the notation from Proposition \ref{Frobenius structure of u-r has $r$-log-decay} and considering the lower left corner, we obtain
		\begin{align*}
			U_iA_{1,1}' &= q A_{2,2}'U_i^{\sigma^f}.
		\end{align*}
		From here it is clear that $U_i=0$. It follows that the $R_i$ and
		$A_{1,1}'$ describe the unit-root sub-$F$-isocrystal $M^{u-r}$
		of $M$. As the connection and Frobenius structure 
		are defined over $\mathcal{O}_{\mathcal{A}^r}$, we see
		that $M^{u-r}$ has $r$-log-decay.

	\end{proof}
		
	\section{Bounded slope theorem}	
		\begin{proposition} \label{DK for polyannuli}
			Let $M$ be a rank $d$ object of $\Fisoc(\mathbb{G}_{m,k}^n\times \mathbb{A}_k^m, \mathbb{A}_k^{n+m})$. Assume
			that the Newton polygon of $M$ remains constant 
			on $x \in \mathbb{G}_{m,k}^n \times \mathbb{A}_k^m$. 
			Let $\eta$ be the generic point of $\mathbb{G}_{m,k}^n\times \mathbb{A}_k^m$. We assume that $a_\eta^{i+1}(M) -a^i_\eta(M)>0$
			and let $M_i$ denote the sub-object of $M$ in
			$\Fisoc(\mathbb{G}_{m,k}^n \times \mathbb{A}_k^m)$
			with $a_\eta^j(M_i)=a_\eta^j(M)$ for $j\leq i$.
			If $a_\eta^{i+1}(M) - a^i_\eta(M)>1$, then $M_i$
			is overconvergent along $\mathbb{A}_k^{n+m}$.
		\end{proposition}
		\begin{proof}
			Assume $a_\eta^{i+1}(M) - a^i_\eta(M)>1$. Let $r=\frac{1}{a_\eta^{i+1}(M) - a^i_\eta(M)}$.
			Since that $M_i$ is overconvergent if and only if
			$\det(M_i)$ is overconvergent, we may replace
			$M$ with $\wedge^{rank(M_i)}M$ twisted so that
			the smallest slope is $0$ and prove that
			the unit-root sub-$F$-isocrystal $M^{u-r}$ is overconvergent.
			By
			Theorem \ref{theorem on log-decay of subcrystal}
			we know that $\epsilon_f(M^{u-r})$ has $r$-log-decay for some $f>0$. As $r<1$,
			we know from Theorem \ref{small log decay $F$-isocrystals are overconvergent} that $\epsilon_f(M^{u-r})$ is overconvergent along $\mathbb{A}_k^{n+m}$. The functor $\epsilon_f$ for
			the corresponding Galois representations corresponds to 
			the composition $\rho: \pi_1(\mathbb{G}_{m,k}^n\times \mathbb{A}_k^m) \to GL_n(\Q_p) \to GL_n(\Q_{p^f})$. It follows that $M^{u-r}$ has
			finite monodromy, and therefore is overconvergent. 
			
		\end{proof}

		\begin{lemma} \label{lifts are finite etale}
			Let $A_0$ and $B_0$ be smooth $k$-algebras. Let
			$f_0:A_0 \to B_0$ be a finite \'etale morphism.
			Let $B$ (resp $A$) be a $p$-adically complete 
			$W(k)$-algebra with
			$B\otimes_{W(k)}k=B_0$ (resp $A\otimes_{W(k)}k=A_0$)
			and let $f:A\to B$ be a lifting of $f_0$. If $B$ 
			are flat then $f$ is finite \'etale.
		\end{lemma}
	\begin{proof}
		Since $f_0$ is finite \'etale, there exists $g_0(x_0) \in A_0[x_0]$
		of degree $d$ such that $B_0=A_0[x_0]/g_0(x_0)$ and $g_0'(x_0)$ is a unit in $B_0$. Let $x \in B$ be a lift of $x_0$. We claim that
		$B$ is isomorphic to $M=A\oplus xA \oplus x^{d-1}A$ as
		an $A$-module via the natural map $M\to B$. It suffices to show
		$\theta_n: M\otimes \Z/p^{n}\Z \to B\otimes \Z/p^n\Z$ is an 
		isomorphism for all $n$. When $n=1$ this is true
		because $f_0$ is finite \'etale. Assume that $\theta_n$
		is an isomorphism. Let $y \in B\otimes \Z/p^{n+1}\Z$.
		We can find $x\in M\otimes \Z/p^n\Z$ such that
		$\theta_n(x)-y \in p^n(B\otimes \Z/p^{n+1})$. Since $\theta_0$
		is an isomorphism we can find $z \in M$ such that
		$p^n\theta_n(z)=\theta_n(x)-y$. This proves surjectivity. 
		The injectivity of $\theta_{n+1}$ follows from the flatness
		assumption. This shows that $B$ is a finite $A$-algebra.
		Furthermore, there exists $g(T) \in A[T]$ of degree $d$
		such that $B=A[T]/g(T)$ and $x$ corresponds to $T$.
		Clearly $g$ reduces to $g_0$ modulo $p$, so we know
		$g'(x)\neq 0$ is a unit in $B$.

			\end{proof}
		\begin{theorem} (Drinfeld-Kedlaya)
			Let $k$ be perfect field of characteristic $p$.
			Let $X$ be a smooth irreducible quasi-compact
			scheme over $k$. Let $M$ be an irreducible
			object of $\Fisoc(X)$ or $\Fisoc^\dagger(X)$.
			Then for each $i$
			\begin{align*}
				|a_{\eta}^{i+1}(M)-a_{\eta}^i(M)|\leq 1.
			\end{align*}
		\end{theorem}
	
		\begin{proof}
			We first take $M$ to be an object of $\Fisoc(X)$. Let $\eta \in X$
			be a generic point. Assume
			that $a_\eta^{i+1}(M) - a^i_\eta(M)>1$. We will show that
			for every closed point $x \in X$, we have
			$b_x^{i}(M)=b_\eta^{i}(M)$, which will imply $M$ is
			not irreducible. By replacing $M$ with a twist of
			$\wedge^i M$, we may assume that $b_\eta^{1}(M)=0$
			and $b_\eta^{2}(M)>1$.
			The de Jong-Oort
			purity theorem (see \cite{dejong-oort}) tells us that the locus in $X$
			where $b_x^{1}>0$ is a closed
			subscheme $D \subset X$ of codimension $1$.
			Let $x_0 \in D$.
			Let $i:C\hookrightarrow  X$ be a smooth curve
			containing $x_0$. We further assume that the
			set theoretic intersection of $D$ and $C$ is equal
			to $\{x_0\}$ and let $U=C-\{x_0\}$. As $U\cap D$
			is empty, there exists a rank one unit-root convergent 
			sub-$F$-isocrystal $M^{u-r}$ contained in $M|_U$. 
			After shrinking $C$, we may find
			a morphism $f:C-\{x_0\} \to \mathbb{G}_m$ that is finite etale
			of degree $d$. Consider $N=f_* M|_U$, which is
			overconvergent and the subcrystal $N^{u-r}=f_* M^{u-r}$. By Proposition \ref{DK for polyannuli}
			we know that $N^{u-r}$ is overconvergent, which implies
			$M^{u-r}$ is overconvergent. Using \cite[Theorem 2.3.7]{kedlaya2011swan} we see that $(M^{u-r})^{\otimes n}$
			extends to all of $C$ for $n\gg 0$. Thus the smallest
			slope of $M^{\otimes n}$ is zero, which means
			$b_{x_0}^{1}=0$. 
			
			Next, let $M$ be an object of $\Fisoc^\dagger(X)$ and
			assume that $a_\eta^{i+1}(M) - a^i_\eta(M)>1$. By the previous 
			paragraph, we know that there is a convergent sub-$F$-isocrystal
			$M_i\subset M$. We claim that $M_i$ is overconvergent.
			As in the previous paragraph, we may
			assume $b_\eta^{2}(M)-b_\eta^{1}(M)>1$ and prove that $M_1$ is
			overconvergent. After twisting, we can assume that $b_\eta^1(M)=0$.  
			Let $f:Y \to X$ be
			a generically etale morphism such that $Y$ has
			a smooth compatification $\bar{Y}$ and
			$E= \bar{Y}-Y$ is a normal crossing divisor. 
			Let $N=f^{*} M$ and $N_1= f^{*}M_1$. 
			For $x \in E$, we let $U \subset \bar{Y}$ be
			an affine neighborhood of $x$ and let $V=U-(U\cap E)$. Let $\mathbb{U}$
			be a smooth lifting of $U$ over $W(k)$ and
			let $\mathcal{U}$ be the rigid fiber of $\mathbb{U}$. Then
			we may regard $N$ as a locally free sheaf $\mathcal{N}$ on a strict neighborhood
			$\mathcal{V}_\epsilon$ of the tube $]V[$ in $\mathcal{U}$ with
			a connection and a compatible Frobenius structure.
			After shrinking $U$, we may assume that $\mathcal{N}$
			is a free $\mathcal{O}_{\mathcal{V}_\epsilon}$-module. 
			By \cite[Theorem 2]{MR2092132} there exists a finite
			\'etale morphism $\pi_0: U \to \mathbb{A}^{m+n}_{k}$ such that
			$\pi(U\cap E)$ is the union of $m$ coordinate hyperplanes.
			Now consider $\pi_*(N)$, which is an object of $\Fisoc^\dagger(\mathbb{G}_{m,k}^{n}\times \mathbb{A}_k^m, \mathbb{A}_k^{n+m})$ and $\pi_*(N_1)$, which is the unit-root
			sub-$F$-isocrystal of $\pi_*(N)$. 
			We may lift $\pi_0$ to a map $\pi: \mathbb{U} \to \Spec(W(k)\ang{T_1,\dots, T_{n+m}})$ and by
			Lemma \ref{lifts are finite etale} we
			see that $\pi$ is finite \'etale. In particular,
			the map $\pi^{rig}$ is finite \'etale and thus
			$(\pi^{rig}|_{\mathcal{V}_\epsilon})_*(\mathcal{N})$ is free.
			In particular, we may regard $\pi_*(N)$ as a free
			$(\sigma,\nabla)$-module over $\mathcal{A}^\dagger$. $(\pi^{rig}|_{\mathcal{V}_\epsilon})_*(\mathcal{N})$.
			By Proposition \ref{DK for polyannuli} we know that $\pi_*(N_1)$ is overconvergent,
			which in turn means $N_1$ is overconvergent when
			restricted to $(V,U)$. Therefore $N_1$ is overconvergent.

			Now let $W$ be a dense open subset of $X$ and $Z \subset Y$
			such that $f: Z \to W$ is a finite \'etale morphism of degree
			$d$. Note that $f_* N_1$ is isomorphic to $(M_1^d)|_W$
			and $f_* N_1$ is overconvergent, which means that $M_1|_W$
			is overconvergent. It follows that $M_1$ is overconvergent.
			 
		\end{proof}

	\bibliographystyle{plain}
	\bibliography{bibliography.bib}
	
\end{document}